\documentclass{amsart}
\usepackage{amsmath}
\usepackage{amssymb}
\usepackage{amsthm}
\usepackage{enumerate}
\usepackage[pdftex]{graphicx}
\usepackage{caption}
\theoremstyle{definition}
\newtheorem{definition}{Definition}[section]
\theoremstyle{plain}
\newtheorem{lemma}[definition]{Lemma}
\newtheorem{theorem}[definition]{Theorem}

\newtheorem{proposition}[definition]{Proposition}

\theoremstyle{remark}
\newtheorem{remark}[definition]{Remark}
\newtheorem{notation}[definition]{Notation}

\makeatletter
\@namedef{subjclassname@2020}{
  \textup{2020} Mathematics Subject Classification}
\makeatother

\newcommand{\mycl}{\operatorname{cl}}
\newcommand{\mystd}{\operatorname{Std}}

\begin{document}
\title[Affiness of one-dimensional definable topological space]{Affiness of topological space definable in a definably complete uniformly locally o-minimal structure of the second kind}
\author[M. Fujita]{Masato Fujita}
\address{Department of Liberal Arts,
Japan Coast Guard Academy,
5-1 Wakaba-cho, Kure, Hiroshima 737-8512, Japan}
\email{fujita.masato.p34@kyoto-u.jp}

\author[T. Kawakami]{Tomohiro Kawakami}
\address{Department of Mathematics,
Wakayama University,
Wakayama, 640-8510, Japan}
\email{kawa0726@gmail.com}

\begin{abstract}
We give a necessary and sufficient condition for a one-dimensional regular and Hausdorff topological space definable in a definably complete uniformly locally o-minimal structure of the second kind having definable bounded multiplication being affine.
\end{abstract}



\maketitle

\section{Introduction}\label{sec:intro}
Metric spaces with metrics definable in o-minimal structures and topological spaces with topologies definable in o-minimal structures have been investigated in \cite{PR,S,V,W}.
Y. Peterzil and A. Rosel gave a necessary and sufficient condition for a one-dimensional topological space with a topology definable in an o-minimal structure being affine in \cite{PR} when the definable set in consideration is bounded.
We tackle the same problem in a different setting.

Definable cell decomposition theorem \cite[Chapter 3, Theorem 2.11]{vdD} is a very useful tool in tackling problems on o-minimal structures.
A locally o-minimal structure \cite{TV} which is defined by simply localizing the definition of o-minimal structures does not admit local definable cell decomposition.
The first author proposed a uniformly locally o-minimal structure of the second kind in \cite{Fuji}, which admits local definable cell decomposition when it is definably complete \cite[Theorem 4.2]{Fuji}. 
A definably complete uniformly locally o-minimal expansion of an ordered abelian group of the second kind is simply called a DCULOAS structure. 
We consider it in place of an o-minimal structure in this paper.

The proof of the implication (4)$\Rightarrow$(1) in \cite[Theorem 4.1]{PR} uses the assumption that arbitrary two closed intervals are definably homeomorphic to each other unless this assumption is not made in the statement.
It is true when the structure is an expansion of an ordered field.
The authors proposed the notion of definable bounded multiplication in \cite{FKK}.
We demonstrated that several assertions which are valid in an o-minimal expansion of an ordered field also hold true in a definably complete locally o-minimal structure having definable bounded multiplication with the extra assumption that the relevant spaces are definably compact in \cite{FKK}.
In this paper, we also employ the additional assumption that the structure has definable bounded multiplication in order to obtain definable homeomorphism between given two closed intervals.

In summary, we consider a topological space with topology definable in a DCULOAS structure having definable bounded multiplication, and find a necessary and sufficient condition for a one-dimensional definable topological space being affine under the assumption that the definable set is bounded.

We next review basic definitions.
\begin{definition}
A densely linearly ordered structure $\mathcal M=(M,<,\ldots)$ is \textit{definably complete} if every definable subset of $M$ has both a supremum and an infimum in $M \cup \{ \pm \infty\}$.

A densely linearly ordered structure $\mathcal M=(M,<,\ldots)$ is \textit{locally o-minimal} if, for every definable subset $X$ of $M$ and for every point $a\in M$, there exists an open interval $I$ containing the point $a$ such that $X \cap I$ is  a finite union of points and open intervals.

A locally o-minimal structure $\mathcal M=(M,<,\ldots)$ is a \textit{uniformly locally o-minimal structure of the second kind} if, for any positive integer $n$, any definable set $X \subset M^{n+1}$, $a \in M$ and $b \in M^n$, there exist an open interval $I$ containing the point $a$ and an open box $B$ containing $b$ such that the definable sets $X_y \cap I$ are finite unions of points and open intervals for all $y \in B$.
\end{definition}

\begin{definition}
An expansion $\mathcal M=(M,<,0,+,\ldots)$ of a densely linearly ordered abelian group has \textit{definable bounded multiplication compatible to $+$} if there exist an element $1 \in M$ and a map $\cdot:M \times M \rightarrow M$ such that
\begin{enumerate}
\item[(1)] the tuple $(M,<,0,1,+,\cdot)$ is an ordered field;
\item[(2)] for any bounded open interval, the restriction $\cdot|_{I \times I}$ of the product $\cdot$ to $I \times I$ is definable in $\mathcal M$.
\end{enumerate}
We simply say that $\mathcal M$ has \textit{definable bounded multiplication} when the addition in consideration is clear from the context. 
The ordered field $(M,<,0,1,+,\cdot)$ is an ordered real closed field by \cite[Proposition 3.3]{FKK}.
\end{definition}

Here is the definition of definable topology.
\begin{definition}
Consider an expansion of a dense linear order and a definable set $X$.
A topology $\tau$ on $X$ is \textit{definable} when $\tau$ has a basis of the form $\{B_y \subseteq X\}_{y \in Y}$, where $\bigcup_{y \in Y} \{y\} \times B_y$ is definable.
 We call the family $\{B_y\}_{y \in Y}$ a \textit{definable basis} of $\tau$.
 The pair $(X,\tau)$ of a definable set and a definable topology on it is called a \textit{definable topological space}.
 
 Since $X$ is a subset of a Cartesian product $M^n$, $X$ has the topology induced from the product topology of $M^n$.
 It is definable and called the \textit{affine topology}.
 The notation $\tau^{\text{af}}$ denotes the affine topology.
\end{definition}

We are now ready to introduce our main result.
Our target in this paper is to demonstrate the following theorem:
\begin{theorem}\label{thm:one-dim_top}
Consider a DCULOAS structure $\mathcal M=(M;<,+,0,\ldots)$ having definable bounded multiplication.
Let $X$ be a definable bounded subset of $M^n$ of dimension one.
Let $\tau$ be a definable topology on $X$ which is Hausdorff and regular.

The following are equivalent:
\begin{enumerate}
\item[(1)] The definable topological space $(X,\tau)$ is definably homeomorphic to a definable subset of $M^k$ with its affine topology for some $k$.
\item[(2)] There is a definable $\tau$-closed and $\tau$-discrete subset $G$ of $X$ at most of dimension zero satisfying the following conditions:
\begin{enumerate}
\item[(i)] The restriction of $\tau$ to $X \setminus G$ coincides with the affine topology on $X \setminus G$;
\item[(ii)] There exists a positive integer $K$ such that, for any $x \in G$ and a definable $\tau$-open neighborhood $U$ of $x$, we can find a definable $\tau$-open neighborhood $V$ of $x$ contained in $U$ such that $V \setminus \{x\}$ has at most $K$ $\tau^{\text{af}}$-definably connected components.
\end{enumerate}
\end{enumerate}
\end{theorem}

Section \ref{sec:preliminary} is a preliminary section.
We prepare several lemmas necessary for the proof of our main theorem.
Section \ref{sec:proof} is the main body of the paper and devoted to the proof of Theorem \ref{thm:one-dim_top}.
Section \ref{sec:appendix} is an appendix.
Peterzil and Rosel's result \cite[Main Theorem]{PR} assumes that the definable set $X$ is bounded.
We extend their result to the non-bounded case in this section.

In the last of this section, we summarize the notations and terms used in this paper.
The term `definable' means `definable in the given structure with parameters' in this paper.
For a linearly ordered structure $\mathcal M=(M,<,\ldots)$, an open interval is a definable set of the form $\{x \in M\;|\; a < x < b\}$ for some $a,b \in M$.
It is denoted by $(a,b)$ in this paper.
We define a closed interval similarly. 
It is denoted by $[a,b]$.
An open box in $M^n$ is the direct product of $n$ open intervals.
Let $A$ be a subset of a topological space.
We call that $A$ is $\tau$-open if it is open under the topology $\tau$.
We also define $\tau$-closed and $\tau$-frontier etc. in the same manner.
The notations $\mycl(A)$ and $\partial A$ denote the closure and the frontier of the set $A$, respectively.
We denote $\mycl^{\tau}(A)$ and $\partial^{\tau} A$ when we emphasize the topology $\tau$.

\section{Preliminary}\label{sec:preliminary}

We prove three lemmas in this section.
We crucially use the local definable cell decomposition theorem in proving the second lemma.  
\begin{lemma}\label{lem:limit}
Consider a DCULOAS structure $\mathcal M=(M;<,+,0,\ldots)$.
Let $\gamma:(0,u] \rightarrow M^n$ be a definable continuous map having the bounded image.
There exists a unique point $x \in M^n$ such that, for any open box containing the point $x$, the intersection $\gamma((0,u]) \cap B$ is not empty.
We denote this point by $\lim_{t \to 0}\gamma(t)$.

In particular, if $C$ is a bounded definable cell of dimension one, the frontier of $C$ consists of two points.
\end{lemma}
\begin{proof}
The image $D=\gamma((0,u])$ is not closed because $\gamma$ is continuous and $(0,u]$ is not closed.
The frontier $\partial D$ of $D$ is not empty.
We have only to demonstrate that $\partial D$ is a singleton.
Since $\gamma([t,u])$ is closed by \cite[Proposition 1.10]{M} for any $0<t<u$, we have $(*):\partial D \subseteq \bigcap_{0<t<u} \mycl(\gamma((0,t)))$. 
Assume for contradiction that $\partial D$ contains two points, say $x$ and $y$.
Take $\varepsilon>0$ so that $2\varepsilon < \max \{|x_i-y_i|\;|\; 1 \leq i \leq n\}$, where $x_i$ and $y_i$ are the $i$-th coordinates of $x$ and $y$, respectively.
Let $B_x$ and $B_y$ be the open boxes whose length are $2\varepsilon$ centered at $x$ and $y$, respectively.
There exists $\delta_x>0$ such that either $(0, \delta_x) \subseteq \gamma^{-1}(B_x)$ or $(0, \delta_x) \cap \gamma^{-1}(B_x)=\emptyset$ by local o-minimality.
The latter equality contradicts the inclusion (*).
We have $(0, \delta_x) \subseteq \gamma^{-1}(B_x)$.
We can also take $\delta_y>0$ so that $(0, \delta_y) \subseteq \gamma^{-1}(B_y)$.
Set $\delta=\min\{\delta_x,\delta_y\}$.
The open interval $(0,\delta)$ is contained in $\gamma^{-1}(B_x \cap B_y)$, which contradicts the fact that $B_x \cap B_y=\emptyset$. 

When $C$ is a bounded  definable cell of dimension one, it is the image of a bounded open interval $(a,b)$ under a definable continuous map $\gamma:(a,b) \rightarrow M^n$ having the bounded image.
Therefore, the frontier of $C$ consists of two points $\lim_{t \to a}\gamma(t)$ and $\lim_{t \to b}\gamma(t)$.
\end{proof}

\begin{notation}
When $\mathcal M=(M;+,0,\ldots)$ is an expansion of an abelian group, $M^n$ is naturally an abelian group. 
For any definable subset $C$ of $M^n$ and a point $a \in M^n$, the notation $a+C$ denotes the set $\{x \in M^n\;|\; x-a \in C\}$.
\end{notation}

\begin{lemma}\label{lem:base}
Consider a DCULOAS structure $\mathcal M=(M;<,+,0,\ldots)$.
Let $X$ and $Z$ be definable subsets of $M^n$ with $\dim X=1$ and $\dim Z=0$.
Let $R$ be a positive element in $M$.
There exist finitely many bounded definable subsets $C_1, \ldots, C_N$ of $M^n$ of dimension one satisfying the following conditions:
\begin{enumerate}
\item[(a)] For any $z \in Z$ and $1 \leq i \leq N$, the intersection $X \cap (z+C_i)$ either coincides with $z+C_i$ or is an empty set;
\item[(b)] There exist $0<u<R$ and definable continuous injective maps $\gamma_i:(0,u) \rightarrow M^n$ such that 
\begin{itemize}
\item the limits $\lim_{t \to 0}\gamma_i(t)$ are the origin,
\item the limits $z+\lim_{t \to u}\gamma_i(t)$ are not in $Z$ for all $z \in Z$ and 
\item the images $\gamma_i((0,u))$ coincide with $C_i$ for all $1 \leq i \leq N$;
\end{itemize}
\item[(c)] The closure of $X \setminus \left(\{z\} \cup \bigcup_{i=1}^N (z+C_i)\right)$ intersects with $z+\mycl(C_i)$ only at the point $z+\lim_{t \to u}\gamma_i(t)$ for any $z \in Z$.
\end{enumerate}
\end{lemma}
\begin{proof}
We first construct a bounded open box $B$ with $Z \cap \mycl(z+B)=\{z\}$.
Let $p_i:M^n \rightarrow M$ be the projection onto the $i$-th coordinate for every $1\leq i \leq n$.
The image $p_i(Z)$ is closed and discrete by \cite[Theorem 1.1]{Fuji3} and \cite[Lemma 2.3]{Fuji4}.
We assume that $p_i(Z)$ has at least three points for every $1\leq i \leq n$.
We can construct $B$ in the same manner in the other cases.
Define the definable functions $g_i:p_i(Z) \setminus \sup p_i(Z) \rightarrow M$ and $h_i:p_i(Z) \setminus \inf p_i(Z) \rightarrow M$ by
\begin{align*}
&g_i(z) = \inf\{x \in Z\;|\; x>z\} - z \text{ and }\\
&h_i(z) = z-\sup\{x \in Z \;|\; x<z\}\text{.}
\end{align*}
The images of $g_i$ and $h_i$ are of dimension zero by \cite[Theorem 1.1]{Fuji3}.
They are discrete and closed by \cite[Lemma 2.3]{Fuji4}.
We can take a positive $d_i \in M$ such that $d_i < \inf g_i(p_i(Z) \setminus \sup p_i(Z))$ and $d_i < \inf h_i(p_i(Z) \setminus \inf p_i(Z))$.
Set $B=\prod_{i=1}^N (-d_i,d_i)$.
We obviously have $Z \cap \mycl(z+B)=\{z\}$ for any $z \in Z$.

Consider the definable set $Y=\bigcup_{z \in Z} \{z\} \times ((-z)+X)$.
It is of dimension one by \cite[Proposition 2.12, Theorem 3.14]{Fuji}.
Let $\pi:M^{2n} \rightarrow M^n$ be the coordinate projection forgetting the first $n$ coordinates.
We have $\dim \pi(Y)=1$ by \cite[Lemma 5.1]{Fuji} and \cite[Theorem 1.1]{Fuji3} because $\pi(Y)$ contains the one-dimensional definable set $(-z)+X$.

Shrink the open box $B$ if necessary.
Then, there exists a definable cell decomposition $\mathcal D=\{D_i\}_{i=1}^L$ of $B$ partitioning the singleton consisting of the origin and $\pi(Y) \cap B$ by \cite[Theorem 4.2]{Fuji}.
Let $\mathcal E$ be the family of the cells in $\mathcal D$ of dimension one contained in $\pi(Y)$ whose closure contains the origin.
Let $E_1, \ldots, E_N$ be the enumeration of the elements in $\mathcal E$.
For all $1 \leq i \leq N$, we can take definable homeomorphisms $\gamma_i:I_i \rightarrow E_i$, where $I_i$ is bounded open intervals, for all $1 \leq i \leq N$ by the definition of cells.
We may assume that $I_i=(0,r_i)$ for some positive $r_i$ and the limit $\lim_{t \to 0}\gamma_i(t)$ is the origin without loss of generality. 
The number of cells are finite and there are only finitely many points contained in the frontier of $1$-dimensional cell by Lemma \ref{lem:limit}.
Taking smaller $r_i$ if necessary, we may assume that $\mycl(D) \cap \gamma_i((0,r_i))=\emptyset$ for any cell $D \in \mathcal D$ of dimension one with $D \neq E_i$.
Furthermore, we may assume that $r_i<R$.

Fix $1 \leq i \leq N$.
Set $Z_i=\{z \in Z\;|\; X \cap (z+C_i) \neq \emptyset\}$.
It is a definable set of dimension zero.
For any $z \in Z_i$, one of the following condition is satisfied because of local o-minimality.
\begin{itemize}
\item There exists a positive $r<r_i$ such that $z+\gamma_i(t) \in X$ for all $0<t<r$;
\item There exists a positive $r<r_i$ such that $z+\gamma_i(t) \not\in X$ for all $0<t<r$.
\end{itemize}
Let $Z_{i1}$ be the set of points in $Z_i$ satisfying the former condition.
The set of points in $Z_i$ satisfying the latter condition is denoted by $Z_{i2}$.
They are obviously definable subsets of $Z_i$ at most of dimension zero.
In this proof, we construct $C_i$ only when both $Z_{i1}$ and $Z_{i2}$ are not empty.
We can construct $C_i$ similarly in the other cases.
We define definable maps $f_{i1}:Z_{i1} \rightarrow M$ and $f_{i2}:Z_{i2} \rightarrow M$ by
\begin{align*}
&f_{i1}(z)=\sup\{t \in M\;|\;0<t<r_i \text{ and }z+\gamma_i(s) \in X\text{ for all }0<s<t\} \text{ and }\\
&f_{i2}(z)=\sup\{t \in M\;|\;0<t<r_i \text{ and }z+\gamma_i(s) \not\in X\text{ for all }0<s<t\}\text{.}
\end{align*}
The images $f_{ij}(Z_{ij})$ are of dimension zero by \cite[Theorem 1.1]{Fuji3} for $j=1,2$.
They are closed and discrete by \cite[Lemma 2.3]{Fuji4}.
The infimums of $f_{ij}(Z_{ij})$ are elements in $f_{ij}(Z_{ij})$, and they are positive.
Set $u=\min\{\inf f_{ij}(Z_{ij})\;|\;1 \leq i \leq N, 1 \leq j \leq 2\}$.
We put $C_i=\gamma_i((0,u))$.
Three conditions (a) through (c) in the lemma are obviously satisfied.
\end{proof}

\begin{lemma}\label{lem:aaa}
Let $\mathcal M=(M,<,0,+,\ldots)$ be an expansion of a densely linearly ordered group having definable bounded multiplication compatible to $+$.
For any $0<u<v$, there exist $0<u'<u<v<v'$ and a definable increasing $\tau^{\text{af}}$-homeomorphism between $[0,u']$ and $[0,v']$.
\end{lemma}
\begin{proof}
Obvious. We omit the proof.
\end{proof}

\section{Proof of main theorem}\label{sec:proof}
We begin to demonstrate the main theorem.
Consider a DCULOAS structure $\mathcal M=(M;<,+,0,\ldots)$.
Let $X$ be a definable subset of $M^n$.
For any definable subset $U$ of $X$, the notation $\mycl^{\text{af}}(U)$ means the closure of $U$ in $M^n$ under the affine topology.
Consider a definable topology $\tau$ on $X$.
The notation $\mathcal B_a$ denotes the definable basis of neighborhoods of $a \in X$.
\begin{definition}
We define the \textit{set of shadows} of a point $a \in X$ to be
\[
\mathbf{S}_{\tau}(a):= \bigcap_{U \in \mathcal B_a}\mycl^{\text{af}}(U)\text{.}
\] 
The set $\mathbf{S}_{\tau}(a)$ is a definable closed subset of $M^n$.
We call a point in $\mathbf{S}_{\tau}(a)$ a \textit{shadow} of $a$.
We simply write $\mathbf{S}(a)$ when the definable topology $\tau$ is clear from the context.
\end{definition}

\begin{proof}[Proof of Theorem \ref{thm:one-dim_top}]
We first demonstrate that the condition (1) implies the condition (2).
Let $f:(X, \tau) \rightarrow (Y,\tau_Y^{\text{af}})$ be a definable homeomorphism, where $Y$ is a definable subset of $M^k$.
Here, the notation $\tau_Y^{\text{af}}$ denotes the affine topology on $Y$.
Consider the affine topology $\tau_X^{\text{af}}$ on $X$.
Let $f':X \rightarrow Y$ be a definable map defined by $f'(x)=f(x)$.
Let $G$ be the set of points at which $f'$ is discontinuous with respect to $\tau_X^{\text{af}}$.
We have $\dim G < \dim X=1$ by \cite[Corollary 1.2]{Fuji3}.
We get $\dim f(G) \leq 0$ by \cite[Theorem 1.1]{Fuji3}.
The set $f(G)$ is $\tau^{\text{af}}$-closed and $\tau^{\text{af}}$-discrete by \cite[Lemma 2.3]{Fuji4}.
The set $G$ is $\tau$-closed and $\tau$-discrete because $f$ is a homeomorphism.

The restriction $f'|_{X \setminus G}$ of $f'$ to $X \setminus G$ induces a homeomorphism between $(X \setminus G, \tau_{X \setminus G}^{\text{af}})$ and $(Y \setminus f(G), \tau_{Y\setminus f(G)}^{\text{af}})$.
We call it $g$.
The composition $h=g^{-1}\circ f|_{X \setminus G}:(X \setminus G, \tau|_{X \setminus G}) \rightarrow (X \setminus G, \tau_{X\setminus G}^{\text{af}})$ is a definable homeomorphism given by $h(x)=x$.
Here, the notation $\tau|_{X \setminus G}$ denotes the topology on $X \setminus G$ induced from $\tau$.
In particular, it implies the condition (2)(i).

Apply Lemma \ref{lem:base} to $Y$ and $f(G)$.
There are finitely many definable subsets $D_1, \ldots, D_K$ of $M^k$ such that the conditions (a) through (c) in Lemma \ref{lem:base} are satisfied.
Let $\eta_i:(0,u') \rightarrow M^n$ be the definable continuous maps given in the condition (b) for all $1 \leq i \leq K$.
Take $x \in G$.
Set $z=f(x)$.
Permuting the sequence $D_1, \ldots, D_K$ if necessary, we may assume that $Y \cap (z+D_i) = z+D_i$ for $1 \leq i \leq L$ and $Y \cap (z+D_i) = \emptyset$ for $L<i \leq K$ by the condition (a).
For any definable $\tau$-open neighborhood $U$ of $x$, the image $U'=f(U)$ is a $\tau^{\text{af}}$-open neighborhood of $z$.
We can take a $\tau^{\text{af}}$-open neighborhood $V'$ of $z$ contained in $U'$ such that $V'$ does not intersect with $Y \setminus (\{z\} \cup \bigcup_{i=1}^K z+D_i)$ and $\eta_i^{-1}(V')$ is of the form $(0,r)$, where $r$ is a sufficiently small positive integer.
The inverse image $V=f^{-1}(V')$ is a $\tau$-open neighborhood of $x$.
The $\tau^{\text{af}}$-definably connected components of $V' \setminus \{z\}$ are of the form  $\eta_i((0,r))$.
Their inverse images $f^{-1}(\eta_i((0,r)))$ are also $\tau^{\text{af}}$-definably connected by the condition (2)(i).
Therefore, $V \setminus \{x\}$ has at most $K$ $\tau^{\text{af}}$-definably connected components.
We have demonstrated that the condition (2)(ii) is satisfied.
\medskip

We next prove the opposite implication.
There is nothing to prove when $G$ is an empty set.
So, we assume that $G$ is not empty.
The frontier of $\partial X$ is at most of dimension zero by \cite[Theorem 5.6]{Fuji} and it is discrete and closed by  \cite[Proposition 2.13, Proposition 3.2]{Fuji4}.
Set $Z=\partial X \cup G$.
Note that $Z$ is discrete and closed in the affine topology.
We demonstrate two claims.
\medskip

\textbf{Claim 1.} $\mathbf{S}(x) \subseteq Z$ for all $x \in G$.
\medskip

Take a point $y \not\in Z$.
Since $\tau$ is Hausdorff, we take a $\tau$-open definable subset $U$ of $X$ such that $x \in U$ and $y \not\in U$.
Since $\tau$ is regular, we get a $\tau$-open definable subset $V$ of $X$ such that $x \in V \subseteq \mycl^{\tau}(V) \subseteq U$.
In particular, we have $y \not\in \mycl^{\tau}(V)$.
By the assumption, we get $\mycl^{\tau}(V) \cap (X \setminus G)=\mycl^{\text{af}}(V) \cap (X \setminus G)$.
In particular, $y \not\in \mycl^{\text{af}}(V)$.
It implies that $y \not\in \mathbf{S}(x)$.
The proof of Claim 1 has finished.
\medskip

\textbf{Claim 2.} Let $A$ be a bounded definable subset such that $\mycl^{\text{af}}(A) \subseteq X \setminus G$.
We have $\mycl^{\text{af}}(A)=\mycl^{\tau}(A)$.
\medskip

The condition (2)(i) implies the inclusion $\mycl^{\text{af}}(A) \subseteq \mycl^{\tau}(A) \subseteq \mycl^{\text{af}}(A) \cup G$.
We have only to demonstrate that $\mycl^{\tau}(A) \cap G$ is an empty set.
Assume for contradiction that we can take a point $x \in \mycl^{\tau}(A) \cap G$.
Consider the definable family $\{\mycl^{\text{af}}(A) \cap \mycl^{\text{af}}(U)\}_{U \in \mathcal B_x}$ of nonempty definable $\tau^{\text{af}}$-closed sets.
This family is obviously a definable filtered collection, which is defined in \cite[Definition 5.5]{FKK}.
Since $\mycl^{\text{af}}(A)$ is definably compact by \cite[Remark 5.6]{FKK}, the intersection $\bigcap_{U \in \mathcal B_x} (\mycl^{\text{af}}(A) \cap \mycl^{\text{af}}(U))$ is not an empty set.
We take a point $y$ in this intersection. 
The relation $y \in \mathbf S(x)$ immediately follows from the definition of the set of shadow points.
It contradicts Claim 1 because $y \in \mycl^{\text{af}}(A) \subseteq X \setminus G$.
We have demonstrated Claim 2.
\medskip

We have $\dim Z=0$ by \cite[Corollary 5.4(ii), Theorem 5.6]{Fuji}.
Let $M^n$ be the ambient space of $X$.
Shifting $X$ if necessary, we may assume that $\mycl^{\text{af}}(X)$ is contained in $(0,R)^n$ for some $R>0$.
Applying Lemma \ref{lem:base} to $X$ and $Z$, we obtain an open box $B$ in $M^n$ containing the origin and finitely many bounded definable subsets $C_1, \ldots, C_N$ of $B$ of dimension one satisfying the following conditions:
\begin{enumerate}
\item[(a)] For any $z \in Z$ and $1 \leq i \leq N$, the intersection $X \cap (z+C_i)$ either coincides with $z+C_i$ or is an empty set;
\item[(b)] There exist $0<u<R$ and definable continuous injective maps $\gamma_i:(0,u) \rightarrow M^n$ such that 
\begin{itemize}
\item the limits $\lim_{t \to 0}\gamma_i(t)$ are the origin,
\item the limits $z+\lim_{t \to u}\gamma_i(t)$ are not in $Z$ for all $z \in Z$ and 
\item the images $\gamma_i((0,u))$ coincide with $C_i$ for all $1 \leq i \leq N$;
\end{itemize}
\item[(c)] The closure of $X \setminus \left(\{z\} \cup \bigcup_{i=1}^N (z+C_i)\right)$ intersects with $z+\mycl^{\text{af}}(C_i)$ only at the point $z+\lim_{t \to u}\gamma_i(t)$ for any $z \in Z$.
\end{enumerate}
By Lemma \ref{lem:aaa}, we may assume that there exists a definable increasing $\tau^{\text{af}}$-homeomorphism $\Phi:[0,u] \rightarrow [0,3nR]$ taking a smaller $u$ and a larger $R$ if necessary.

The notation $C_{i,z}$ denotes the set $z+C_i$ for every $1 \leq i \leq N$ and $z \in Z$.
We say that $C_{i,z}$ is \textit{inclusive} if $X \cap C_{i,z}=C_{i,z}$, and we call it \textit{exclusive} if $X \cap C_{i,z}=\emptyset$.
The definable set $C_{i,z}$ is either inclusive or exclusive by the condition (a).
By the condition (b) and Lemma \ref{lem:limit}, the frontier $\partial C_{i,z}$ of $C_{i,z}$ under the affine topology consists of two points.
One is the point $z$.
Another point is called the \textit{non-trivial endpoint} of $C_{i,z}$.
The map $$p_i:Z \rightarrow M^n$$ is the definable map so that $p_i(z)$ is the non-trivial endpoint of $C_{i,z}$ for all $1 \leq i \leq N$.
Note that the map $p_i$ is injective.
By the condition (b), the non-trivial endpoint of inclusive $C_{i,z}$ lies in $X \setminus G$.

Consider an inclusive $C_{i,z}$.
The notation $\partial^{\tau}C_{i,z}$ denotes the $\tau$-frontier $\mycl^{\tau}(C_{i,z}) \setminus C_{i,z}$.
We investigate the $\tau$-frontier $\partial^{\tau}C_{i,z}$ of $C_{i,z}$.
\medskip

\textbf{Claim 3.} The intersection $(\partial^{\tau}C_{i,z}) \cap G$ has at most one point.
\medskip

Consider the case in which $(\partial^{\tau}C_{i,z}) \cap G$ is not empty.
Take a point $x$ in this set.
We have only to demonstrate that $y \not\in \mycl^{\tau}C_{i,z}$ for $y \in G$ with $y \neq x$.
Since $\tau$ is Hausdorff, we can take a $\tau$-open definable subset $U$ of $X$ with $x \in U$ and $y \not\in U$.
Since $\tau$ is regular, we can take a $\tau$-open definable subset $V$ of $X$ such that $x \in V \subseteq \mycl^{\tau}(V) \subseteq U$.

By local o-minimality, we have either $z+\gamma_i((0,r)) \subseteq V$ or $(z+\gamma_i((0,r))) \cap V = \emptyset$ for any sufficiently small $r>0$.
If the latter holds true, we have $x \in \mycl^{\tau}(z+\gamma_i([r,u)))$.
On the other hand, we get $\mycl^{\tau}(z+\gamma_i([r,u)))=\mycl^{\text{af}}(z+\gamma_i([r,u)))$ by Claim 2.
It implies $x \in \mycl^{\text{af}}(z+\gamma_i([r,u)))$, which is a contradiction to the condition (b).

Take $r>0$ so that $z+\gamma_i((0,r)) \subseteq V$.
We finally get $\mycl^{\tau}(C_{i,z})=\mycl^{\tau}(z+\gamma_i((0,r))) \cup \mycl^{\tau}(z+\gamma_i([r,u))) \subseteq \mycl^{\tau}(V) \cup \mycl^{\text{af}}(z+\gamma_i([r,u)))$ using Claim 2 again. 
It implies that $y \not\in \mycl^{\tau}C_{i,z}$ because $y \in G$ and $G \cap \mycl^{\text{af}}(z+\gamma_i([r,u)))=\emptyset$.
The proof of Claim 3 has been completed.
\medskip

Thanks to Claim 3, $(\partial^{\tau}C_{i,z}) \cap G$ has a unique point when it is not empty.
This unique point is called the \textit{$\tau$-connection point} of $C_{i,z}$.
We set
\begin{align*}
\mathcal Z_{\text{no},i}=\{z \in Z\;|\; &C_{i,z}\text{ is inclusive and does not have its }\tau\text{-connection point}\}\text{,}\\
\mathcal Z_{\text{self},i}=\{z \in Z\;|\; &C_{i,z}\text{ is inclusive, and it has its }\tau\text{-connection point}\\
& \text{and its }\tau\text{-connection point} = z\}\\
\mathcal Z_{i}=\{z \in Z\;|\; &C_{i,z}\text{ is inclusive, and it has its }\tau\text{-connection point}\\
& \text{and its }\tau\text{-connection point} \neq z\}
\end{align*}
for $1 \leq i \leq N$.
The sets $\mathcal Z_{\text{no},i}$, $\mathcal Z_{\text{self},i}$ and $\mathcal Z_i$ are definable.
We also define the definable maps $$\zeta_i:\mathcal Z_{i} \cup \mathcal Z_{\text{self},i} \rightarrow G$$ so that $\zeta_i(z)$ is the $\tau$-connection point of $C_{i,z}$.

Fix $1 \leq i \leq N$ and $z \in \mathcal Z_i$.
Let $\overline{\gamma}_{i,z}:[0,u] \rightarrow M^n$ be the $\tau^{\text{af}}$-continuous extension of the definable $\tau^{\text{af}}$-continuous curve $t \mapsto z+\gamma_i(t)$ for all $1 \leq i \leq N$ and $z \in \mathcal Z_i$.
Such extension exists by Lemma \ref{lem:limit}.
On the other hand, for any two points $x=(x_1,\ldots, x_n)$ and $y=(y_1,\ldots, y_n)$, we set $$d(x,y)=\sum_{i=1}^n|x_i-y_i|\text{.}$$
We set $d'_{i,z}(t)=d(\zeta_i(z),\overline{\gamma}_{i,z}(t))+2d(\overline{0}_n,\overline{\gamma}_{i,z}(t))$ for simplicity, where $\overline{0}_n$ denotes the origin of $M^n$.
The function $d'_{i,z}$ is definable and $\tau^{\text{af}}$-continuous.
We get $\Phi(0)=0<d'_{i,z}(0)=d(\zeta_i(z),z)+2d(\overline{0}_n,z)>0=\Phi(0)$ and $\Phi(u)=3nR>d(\zeta_i(z),p_i(z))+2d(\overline{0}_n,p_i(z))=d'_{i,z}(u)$.
By the intermediate value theorem, we obtain $\Phi(v)=d'_{i,z}(v)$ for some $0<v<u$.
Therefore, the map $v_i:\mathcal Z_i \rightarrow (0,u)$ given by
\begin{align*}
& v_i(z)=\inf\{0<t<u\;|\; \Phi(t)=d'_{i,z}(t)\}
\end{align*}
is a well-defined definable function.
The definable map $q_i:\mathcal Z_i \rightarrow \bigcup_{z \in \mathcal Z_i} C_{i,z}$ is given by $q_i(z)=z+\gamma_i(v_i(z))$.
We get $q_i(z) \in C_{i,z}$ and $$d(\zeta_i(z),q_i(z))+2d(\overline{0}_n,q_i(z))=\Phi(v_i(z))\text{.}$$
We put
$$C'_{i,z}=z+\gamma((0,v_i(z)))\text{.}$$

We also need the following claim.
\medskip

\textbf{Claim 4.} The inverse images $\zeta_i^{-1}(x)$ have at most $K$ points for all $1 \leq i \leq N$ and $x \in G$.
\medskip

Assume for contradiction that $\zeta_i^{-1}(x)$ has $K+1$ points $z_1, \ldots, z_{K+1}$.
Consider the definable $\tau^{\text{af}}$-closed subsets $z_j+\gamma_i([u/3,2u/3])$ for $1 \leq j \leq K+1$.
It is also $\tau$-closed by Claim 2.
There exists a definable $\tau$-open neighborhood $U$ of $x$ in $X$ such that $U \cap  (z_j+\gamma_i([u/3,2u/3]))=\emptyset$ for all $0 \leq j \leq K+1$ and $U \setminus \{x\}$ has at most $K$ $\tau^{\text{af}}$-definably connected components by the condition (2)(ii).
On the other hand, the set $z_j + \gamma_i((0,r_j))$ is a $\tau^{\text{af}}$-definably connected component of $U \setminus \{x\}$ for some sufficiently small $r_j$ for any $1 \leq j \leq K+1$ because of the definition of $\tau$-connection points.
Contradiction.
We have proved Claim 4.

\medskip
The inverse image $\zeta_i^{-1}(x)$ consists of at most $K$ points by Claim 4.
We set 
\begin{align*}
\mathcal Z_{i,k}=\{z \in \mathcal Z_i\;|\; &z \text{ is the }k\text{-th smallest element in }\zeta_i^{-1}(\zeta_i(z))\\
& \text{ in the lexicographic order}\}
\end{align*}
for $1 \leq i \leq N$ and $1 \leq k \leq K$.
It is also definable.
We have $\mathcal Z_i = \bigcup_{k=1}^K \mathcal Z_{i,k}$.
We also put 
\begin{align*}
&\mathcal C_{\text{no},i} = \bigcup_{z \in \mathcal Z_{\text{no},i}}C_{i,z} \text{,} \qquad
\mathcal C_{\text{self},i} = \bigcup_{z \in \mathcal Z_{\text{self},i}}C_{i,z} \text{,} \qquad
\mathcal C'_{i,k} = \bigcup_{z \in \mathcal Z_{i,k}}C'_{i,z} \text{ and }\\
&\mathcal C_i = \mathcal C_{\text{no},i} \cup \bigcup_{k=1}^K \mathcal C'_{i,k}\text{.}
\end{align*}
We can define the definable map $\rho_{i}:\mathcal C_{i} \rightarrow  Z$ so that $\rho_{i}(x)=z$ if and only if $x \in C_{i,z}$.

We need an extra preparation.
For given points $p,q \in M^n$, the \textit{standard connection between $p$ and $q$} is  the union of $n$ segments connecting the $n+1$ points $p=v_0, v_1,\ldots, v_n=q$ in the given order, where $v_j$ is the point whose last $j$ coordinates equal the last $j$ coordinates of $q$ and whose first $n-j$ coordinates are the first $n-j$ coordinates of $p$.
We denote it by $\mystd(p,q)$.
We can easily construct a definable map $\Psi:\{(t,p,q) \in M \times M^n \times M^n \;|\; 0 \leq t \leq d(p,q)\} \rightarrow M^n$ such that, for fixed $p,q \in M^n$, the restriction $\Psi(\cdot,p,q):[0,d(p,q)] \rightarrow M^n$ has the image $\mystd(p,q)$  and this restriction is a $\tau^{\text{af}}$-homeomorphism onto its image. 
\medskip

We have now finished long preparation.
We construct a definable subset $Y$ of $M^{n(1+KN)}$ and a definable homeomorphism $f:(X,\tau) \rightarrow (Y,\tau^{\text{af}})$.
The notation $\overline{0}_m$ denotes the origin of $M^m$ for any positive integer $m$.

We define that $Y$ is the union of the following subsets of $M^{n(1+KN)}$:
\begin{itemize}
\item $(X \setminus \bigcup_{i=1}^N \mathcal C_i)\times \{\overline{0}_{nKN}\}$;
\item For any $1 \leq i \leq N$ and $z \in Z$ such that $C_{i,z}$ is inclusive, 
\begin{itemize}
\item $\{p_i(z)\} \times \{\overline{0}_{(i-1)nK}\} \times (0,u) \times \{\overline{0}_{nK(N-i+1)+n-1}\}$ for $z \in \mathcal Z_{\text{no},i}$;
\item when $z \in \mathcal Z_{i,k}$, the union of the following three definable sets:
\begin{itemize}
\item $\{\zeta_i(z)\} \times \{\overline{0}_{m_1}\} \times \mystd(\overline{0}_n,q_{i}(z)) \times \{\overline{0}_{m_2}\}$, 
\item $\mystd(\zeta_i(z),q_{i}(z)) \times \{\overline{0}_{m_1}\} \times \{q_{i}(z)\} \times \{\overline{0}_{m_2}\}$ and 
\item $\{q_{i}(z)\} \times \{\overline{0}_{m_1}\} \times \mystd(q_{i}(z),\overline{0}_n) \times \{\overline{0}_{m_2}\}$,
\end{itemize}
where $m_1=n((i-1)K+(k-1))$ and $m_2=n(KN-iK+K-k)$.
\end{itemize}
\end{itemize}
Roughly speaking, when $z \in \mathcal Z_i$, we connect the non-trivial endpoint $p_i(z)$ and the $\tau$-connection point $\zeta_i(z)$ with a curve for any $1 \leq i \leq N$ and $z \in Z$.
Since at most $KN$ curves whose $\tau^{\text{af}}$-closures do not contain the point $z$ gather at any point $z$ in $G$ under the topology $\tau$, we can connect them so that two curves do not intersect each other.
We can construct a first order formula defining the set $Y$ using $p_i$, $\zeta_i$ and $C_i$.
It implies that $Y$ is definable.

We next construct the definable homeomorphism $f:(X,\tau) \rightarrow (Y,\tau^{\text{af}})$.
We define $f(x)$ as follows:
\begin{itemize}
\item $f(x)=(x,\overline{0}_{nKN})$ when $x$ is in $X \setminus \bigcup_{i=1}^N \mathcal C_i$;
\item When $x$ is contained in $\mathcal C_{i}$, we define as follows:
\begin{itemize}
\item When $x \in \mathcal C_{\text{no},i}$, we set 
$
f(x)=(p_i(\rho_i(x)),\overline{0}_{m'_1},u-\gamma_i^{-1}(x-\rho_{i}(x)),\overline{0}_{m'_2})\text{,}
$
where $m'_1=(i-1)nK$ and $m'_2=nK(N-i+1)+n-1$.
\item We consider the case in which $x \in \mathcal C_{i,k}$, set $z=\rho_{i,k}(x)$, $t=\gamma_i^{-1}(x-z)$, $\mathfrak q=q_i(z)$ and $\zeta=\zeta_i(z)$.
We define as follows:
\begin{itemize}
\item $f(x)=(\zeta,\overline{0}_{m_1},\Psi(t,\overline{0}_n,\mathfrak q),\overline{0}_{m_2})$ when $0<t \leq d(\overline{0}_n,\mathfrak q)$;
\item $f(x)=(\Psi(t-d(\overline{0}_n,\mathfrak q),\zeta,\mathfrak p) ,\overline{0}_{m_1},\mathfrak q,\overline{0}_{m_2})$ when $d(\overline{0}_n,\mathfrak q)<t \leq d(\overline{0}_n,\mathfrak q)+d(\zeta,\mathfrak q)$;
\item $f(x)=(\mathfrak q,\overline{0}_{m_1},\Psi(t-d(\overline{0}_n,\mathfrak q)-d(\zeta,\mathfrak q),\mathfrak q,{0}_n) ,\overline{0}_{m_2})$ when $d(\overline{0}_n,\mathfrak q)+d(\zeta,\mathfrak q)<t<v_i(z)$,
\end{itemize}
where $m_1=n((i-1)K+(k-1))$ and $m_2=n(KN-iK+K-k)$.
\end{itemize} 
\end{itemize}
The map $f$ defined above is a definable map.
It is not difficult to prove that $f$ is a bijection.
The proof is left to the readers.
The remaining task is to show that $f$ is a homeomorphism.
Note that the restriction of $f$ to $X \setminus G$ induces a definable homeomorphism between $(X \setminus G,\tau_{\text{af}})$ and $(Y \setminus f(G),\tau_{\text{af}})$. 

We prove that $f$ is an open map.
Take a $\tau$-open definable set $U$ of $X$.
Set $V=f(U)$.
Take a point $y \in V$ and set $x=f^{-1}(y)$.
We construct a definable open neighborhood $W$ of $y$ contained in $V$.
The case in which $x \not\in G$ is easy.
The set $U \setminus G$ is $\tau^{\text{af}}$-open by the assumption.
Take a $\tau^{\text{af}}$-open neighborhood $U_1$ of $x$ in $X$ contained in $U \setminus G$.
It is also $\tau$-open because $G$ is $\tau$-closed.
Since the restriction of $f$ to $X \setminus G$ induces a definable homeomorphism, the image $W=f(U_1)$ is an $\tau^{\text{af}}$-open neighborhood of $y$.

When $x \in G$, take all $C_{i,z}$ whose $\tau$-connection point is $x$.
We can take finitely many such sets by Claim 4.
Set $\mathcal D(x)=\{(i,z) \in \mathbb Z \times Z\;|\; 1 \leq i \leq N, \ z \in \mathcal Z_{i} \cup \mathcal Z_{\text{self},i}, \ \zeta_i(z)=x\}$.
For any $(i,z) \in \mathcal D(x)$, there exists $u_{i,z}>0$ such that $z+\gamma_i((0,u_{i,z}))$ is contained in $U$ by the definition of $\tau$-connection points.
We can choose $u_{i,z}$ so that the map $z \mapsto u_{i,z}$ is definable.
By the definition of $Y$ and $f$, $$W=\{y\} \cup \bigcup_{(i,z) \in \mathcal D(x)} f(\gamma_i((0,u_{i,z})))$$ is a definable $\tau^{\text{af}}$-open subset of $V$ containing the point $y$.

We next prove that the inverse $f^{-1}$ of $f$ is an open map.
Take a $\tau^{\text{af}}$-open definable set $V$ of $Y$.
Set $U=f^{-1}(V)$.
Take a point $x \in U$.
We construct a definable open neighborhood $U_1$ of $x$ contained in $U$.
The case in which $x \not\in G$ is easy.
We omit the proof.

We consider the case in which $x \in G$.
Since $G$ is $\tau$-closed and $\tau$-discrete, the set $G \setminus \{x\}$ is $\tau$-closed.
We can take disjoint definable $\tau$-open subsets $U_1$ and $U_2$ of $X$ such that $x \in U_1$ and $G \setminus \{x\} \subseteq U_2$ because $\tau$ is regular.
We show that $(z+\gamma_i((0,t))) \cap U_1=\emptyset$ for all sufficiently small $t>0$ when $1 \leq i \leq N$, $z \in \mathcal Z_{i}$ and $\zeta_i(z) \neq x$.
Assume the contrary.
There exists $1 \leq i \leq N$ and $z \in \mathcal Z_{i}$ such that $\zeta_i(z) \neq x$ and $z+\gamma_i((0,t)) \subseteq U_1$ for some $t>0$ by local o-minimality.
The set $z+\gamma_i((0,t))$ has a nonempty intersection with any $\tau$-neighborhood of $\zeta_i(z)$ by the definition of a $\tau$-connection point.
In particular, we have $U_1 \cap U_2 \neq \emptyset$, which is a contradiction.

We next consider the intersection of $U_1$ with $C_{i,z}$ for $z \in \mathcal Z_{\text{no},i}$.
Shrinking $U_1$ if necessary, we may assume that $U_1$ has at most $K$ $\tau^{\text{af}}$-definably connected components by the condition (2)(ii).
In particular, we have only finitely many $z \in \mathcal Z_{\text{no},i}$ such that $U_1$ contains $z+\gamma_i((0,t))$ for some $t>0$.
Since $z+\gamma_i((0,t/2])$ is $\tau$-closed by Claim 3, we may assume that, for any $z \in \mathcal Z_{\text{no},i}$, $(z+\gamma_i((0,s))) \cap U_1=\emptyset$ for all sufficiently small $s>0$ by removing $z+\gamma_i((0,t/2])$ from $U_1$ if necessary.

We set $\mathcal Z_{\text{in},i}(x)=\{z \in \mathcal Z_{i} \cup \mathcal Z_{\text{self},i}\;|\; \zeta_i(z) = x\} $ and $\mathcal Z_{\text{out},i}(x)=\mathcal Z_{\text{no},i} \cup \{z \in \mathcal Z_{i} \cup \mathcal Z_{\text{self},i}\;|\; \zeta_i(z) \neq x\}$.
We also consider the definable map $\xi_i:\mathcal Z_{\text{out},i}(x) \rightarrow M$ given by 
$$
\xi_i(z)=\sup\{0<t<u\;|\; (z+\gamma_i((0,t))) \cap U=\emptyset\}\text{.}
$$ 
Consider the definable set $$\Pi=X \setminus \left(D \cup \bigcup_{i=1}^N\left(\bigcup_{z \in \mathcal Z_{\text{out},i}(x)}(z+\gamma((0,\xi_i(z)))) \cup \bigcup_{z \in \mathcal Z_{\text{in},i}(x)} C_{i,z}\right)\right)\text{.}$$
It is a definable $\tau^{\text{af}}$-closed subset of $X \setminus D$.
It is also $\tau$-closed by Claim 2. 
Removing $\Pi$ from $U_1$, we may assume that $U_1$ is contained in $D \cup \bigcup_{i=1}^N\bigcup_{z \in \mathcal Z_{\text{in},i}(x)} C_{i,z}\text{.}$
As we demonstrated previously, $D \setminus \{x\}$ is $\tau$-closed.
Removing it from $U_1$, we may assume that $U_1$ is contained in $\{x\} \cup \bigcup_{i=1}^N\bigcup_{z \in \mathcal Z_{\text{in},i}(x)} C_{i,z}\text{.}$

Take a sufficiently small $r>0$ so that $z+\gamma_i((0,r)) \subseteq U$ for any $1 \leq i \leq N$ and $z \in \mathcal Z_{\text{in},i}(x)$.
The definable set $\bigcup_{i=1}^N \bigcup_{z \in \mathcal Z_{\text{in},i}(x)}(z+\gamma([r,u)) \cup p_i(z))$ is $\tau^{\text{af}}$-closed.
It is also $\tau$-closed by Claim 2.
Removing this set from $U_1$, we may assume that $U_1$ is contained in $U$.
We have finally constructed a definable $\tau$-open neighborhood $U_1$ of $x$ contained in $U$.
\end{proof}

\begin{remark}
We employ the assumption that $\mathcal M$ has definable bounded multiplication in the theorem so as to construct a definable homeomorphism given in Lemma \ref{lem:aaa}.
Therefore, when we can construct this homeomorphism for some other reason, we do not need this assumption.
\end{remark}

\begin{remark}
Consider the following condition:
\begin{itemize}
\item There exists a positive integer $K$ such that $|\{y \in X\;|\; x \in \mathbf S(y)\}| \leq K$ for all $x \in G$.
\end{itemize}
The condition (2)(ii) implies this condition as demonstrated in Claim 4, but the converse is not true at least when $X$ is not bounded.
Consider the locally o-minimal structure $(\mathbb R;<,+,0,\mathbb Z)$ given in \cite[Example 20]{KTTT}.
Consider the definable topology $\tau$ on $\mathbb R$ whose open basis $\mathcal B$ is given below:
For any $x \neq 0$, we set $\mathcal B_x=\{(x-r,x+r)\;|\; 0<r \in \mathbb R\}$.
We also set $\mathcal B_0=\{ (-r,r) \cup (\mathbb Z_{>n})\;|\; 0<r \in \mathbb R,\ n \in \mathbb Z\}$, where $\mathbb Z_{>n}:=\{x \in \mathbb Z\;|\; x>n\}$.
We have $\mathbf S(x)=\{x\}$ for all $x \in \mathbb R$, and the above condition is satisfied.
However, the condition (2)(ii) fails when $x=0$ because any definable $\tau$-open neighborhood of the origin has infinitely many $\tau^{\text{af}}$-definably connected components of the form $\{m\}$ with $m \in \mathbb Z_{>n}$ for some $n \in \mathbb Z$.
\end{remark}

\section{Appendix: extension of Peterzil and Rosel's result to non-bounded case}\label{sec:appendix}

Y. Peterzil and A. Rosel gave a necessary and sufficient condition for a one-dimensional topological space with a topology definable in an o-minimal structure being affine in \cite{PR} when the definable set in consideration is bounded.
We extend this result to the non-bounded case in this appendix.
We first review their main theorem.

\begin{theorem}[{\cite[Main theorem]{PR}}]\label{thm:pr}
Let $\mathcal M=(M;<.0,+,\ldots)$ be an o-minimal expansion of an ordered group.
Assume that arbitrary two closed intervals are definably homeomorphic.
Let $X \subseteq M^n$ be a definable bounded set with $\dim X=1$, and let $\tau$ be a definable Hausdorff topology on $X$.
Then the following are equivalent:
\begin{enumerate}
\item[(1)] $(X,\tau)$ is definably homeomorphic to a definable subset of $M^k$ for some $k$, with its affine topology.
\item[(2)] There is a finite set $G \subseteq X$ such that every $\tau$-open subset of $X \setminus G$ is open with respect to the affine topology on $X \setminus G$.
\item[(3)] Every definable subset of $X$ has finitely many definably connected components, with respect to $\tau$.
\item[(4)] $\tau$ is regular and $X$ has finitely many definably connected components with respect to $\tau$. 
\end{enumerate}
\end{theorem}

The assumption that $X$ is bounded could be omitted when there exists a definable bijection between a bounded interval and an unbounded interval.
The structure not satisfying the above condition is investigated in \cite{E}. 
It is called a \textit{semi-bounded} o-minimal structure.
Theorem \ref{thm:pr} is not true if we omit the assumption that $X$ is bounded as in the example in \cite[Section 4.3]{PR}.
In the non-bounded case, we get the following proposition:

\begin{proposition}\label{prop:pr2}
Let $\mathcal M=(M;<.0,+,\ldots)$ be a semi-bounded o-minimal expansion of an ordered group.
Assume that arbitrary two closed intervals are definably homeomorphic.
Let $X \subseteq M^n$ be a definable set with $\dim X=1$, and let $\tau$ be a definable Hausdorff topology on $X$.
Then the following are equivalent:
\begin{enumerate}
\item[(1)] $(X,\tau)$ is definably homeomorphic to a definable subset of $M^k$ for some $k$, with its affine topology.
\item[(2)] There is a finite set $G \subseteq X$ such that the restriction of $\tau$ to $X \setminus G$ coincides with the affine topology on $X \setminus G$.
\end{enumerate}
\end{proposition}
\begin{proof}
When $X$ is bounded, the proposition follows from Theorem \ref{thm:pr}.
Therefore, we only treat the case in which $X$ is not bounded.

We first demonstrate that the condition (1) implies the condition (2).
Let $f:(X, \tau) \rightarrow (Y,\tau_Y^{\text{af}})$ be a definable homeomorphism, where $Y$ is a definable subset of $M^k$.
Here, the notation $\tau_Y^{\text{af}}$ denotes the affine topology on $Y$.
Consider the affine topology $\tau_X^{\text{af}}$ on $X$.
Let $f':X \rightarrow Y$ be a definable map defined by $f(x)=f'(x)$.
Let $G$ be the set of points at which $f'$ is discontinuous with respect to $\tau_X^{\text{af}}$.
It is well-known that $\dim G < \dim X=1$, so $G$ is a finite set.
The restriction $f'|_{X \setminus G}$ of $f'$ to $X \setminus G$ induces a homeomorphism between $(X \setminus G, \tau_{X\setminus G}^{\text{af}})$ and $(Y \setminus f(G), \tau_{Y\setminus f(G)}^{\text{af}})$.
We call it $g$.
The composition $h=g^{-1}\circ f|_{X \setminus G}:(X \setminus G, \tau|_{X \setminus G}) \rightarrow (X \setminus G, \tau_{X\setminus G}^{\text{af}})$ is a definable homeomorphism given by $h(x)=x$.
Here, the notation $\tau|_{X \setminus G}$ denotes the topology on $X \setminus G$ induced from $\tau$.
In particular, it implies the condition (2).

We next prove the opposite implication.
Since $G$ is finite, there exists a bounded open box $B$ in $M^n$ containing the set $G$.  
We get a stratification of $\mycl(X)$ partitioning $X$, $G$ and $X \cap B$ by \cite[Chapter 4, Proposition 1.13]{vdD}.
Let $X'$ be the union of bounded cells in the stratification contained in $X$.
It is a bounded definable set.
There exist a definable subset $Y'$ of $M^l$ for some $l$ and a definable homeomorphism $f':(X',\tau|_{X'}) \rightarrow (Y',\tau_{Y'}^{\text{af}})$ by Theorem \ref{thm:pr}.

Let $C_1, \ldots, C_N$ be the unbounded cells of the stratification contained in $X$.
Note that the topology on $C_i$ induced from $\tau$ coincides with the affine topology by the definition of the stratification. 
The cells $C_i$ are the graphs of definable continuous maps $\varphi_i:I_i \rightarrow M^n$, where $I_i$ are open intervals, and the frontier $\partial C_i$ is either a finite set or an empty set for all $1 \leq i \leq N$. 
We may assume that either $I_i=(0,\infty)$ or $I_i=M$ without loss of generality.
Since $C_i$ is unbounded $X \cap \partial C_i$ is at most one point.
We may assume that $X \cap \partial C_i \neq \emptyset$ for all $1 \leq i \leq L$ and $X \cap \partial C_i = \emptyset$ for all $L < i \leq N$.
Let $x_i$ be the unique points in $X \cap \partial C_i$ for all $1 \leq i \leq L$.
For $L < i \leq N$, take distinct points $x_i$ in $M^l$ out of $\mycl(B)$.

Set $k=l+1$.
We construct a definable subset $Y$ of $M^k$ and a definable homeomorphism $f:(X,\tau) \rightarrow (Y,\tau_{Y}^{\text{af}})$.
We set $$Y= (Y' \times \{0\}) \cup \bigcup_{i=1}^N \{x_i\} \times I_i\text{.}$$
The map $f:X \rightarrow Y$ is defined as follows.
We set $f(x)=(f'(x),0)$ when $x \in X'$.
We set $f(x)=(x_i,\varphi_i^{-1}(x))$ when $x \in C_i$ for some $1 \leq i \leq N$.
It is a routine to demonstrate that $f$ is a homeomorphism.
\end{proof}


\begin{thebibliography}{99}


%
\bibitem{vdD}
L. van den Dries, 
\emph{Tame topology and o-minimal structures},
London Mathematical Society Lecture Note Series, Vol. 248.
Cambridge University Press, Cambridge, 1998.
%
\bibitem{E}
M. J. Edmundo, 
\emph{Structure theorems for o-minimal expansions of groups},
Ann. Pure Appl. Logic, \textbf{102} (2000), 159-181.
%
\bibitem{Fuji}
M. Fujita,
\emph{Uniformly locally o-minimal structures and locally o-minimal structures admitting local definable cell decomposition},
Ann. Pure Appl. Logic, \textbf{171} (2020), 102756.
%

\bibitem{Fuji3}
M. Fujita,
\emph{Dimension inequality for a definably complete uniformly locally o-minimal structure of the second kind},
J. Symbolic Logic, \textbf{85} (2020), 1654-1663.

\bibitem{Fuji4}
M. Fujita,
\emph{Locally o-minimal structures with tame topological properties},
J. Symbolic Logic, to appear.
%

\bibitem{FKK}
M. Fujita, T. Kawakami and W. Komine,
\emph{Tameness of definably complete locally o-minimal structures and definable bounded multiplication},
preprint (2021).

%

\bibitem{KTTT}
T. Kawakami, K. Takeuchi, H. Tanaka and A. Tsuboi,
\emph{Locally o-minimal structures},
J. Math. Soc. Japan, \textbf{64} (2012), 783-797.


\bibitem{M}
C. Miller,
\emph{Expansions of dense linear orders with the intermediate value property},
J. Symbolic Logic, \textbf{66} (2001), 1783-1790.

\bibitem{PR}
Y. Peterzil and A. Rosel,
\emph{Definable one-dimensional topologies in o-minimal structures},
Archiv Math. Logic, \textbf{59} (2020), 103-125.


\bibitem{TV}
C. Toffalori and K. Vozoris, 
\emph{Notes on local o-minimality},
MLQ Math. Log. Q., \textbf{55} (2009), 617-632.

\bibitem{S}
M. Shiota, 
\emph{Semialgebraic metric spaces and resolution of singularities of definable sets},
arXiv:1701.03858 (2017).

\bibitem{V}
G. Vallette, 
\emph{On metric types that are definable in an o-minimal structure},
J. Symbolic Logic, \textbf{73} (2008), 439-447.

\bibitem{W}
E. Walsberg, 
\emph{On the topology of metric spaces definable in o-minimal expansions of fields},
arXiv:1510.07291 (2015).

\end{thebibliography}
\end{document}